\documentclass[12pt,dvips]{amsart}
\usepackage{amsfonts, amssymb, latexsym, epsfig, amscd}

\newtheorem{Theorem}{Theorem}[section]

\newtheorem{Lemma}[Theorem]{Lemma}
\newtheorem{Claim}[Theorem]{Claim}

\newtheorem{Corollary}[Theorem]{Corollary}

\newtheorem{Definition-Proposition}[Theorem]{Definition-Theorem}
\newtheorem{Main Conjecture}[Theorem]{Main Conjecture}

\theoremstyle{remark}
\newtheorem{Example}[Theorem]{Example}

\newcommand{\id}{\mathrm{id}}

\newcommand\iso{{\cong}}

\theoremstyle{plain}

\newcommand{\cellsize}{19}
\newlength{\cellsz} 
\newsavebox{\cell}
\sbox{\cell}{\begin{picture}(\cellsize,\cellsize)
\put(0,0){\line(1,0){\cellsize}}
\put(0,0){\line(0,1){\cellsize}}
\put(\cellsize,0){\line(0,1){\cellsize}}
\put(0,\cellsize){\line(1,0){\cellsize}}
\end{picture}}
\newcommand\cellify[1]{\def\thearg{#1}\def\nothing{}%
\ifx\thearg\nothing
\vrule width0pt height\cellsz depth0pt\else
\hbox to 0pt{\usebox{\cell} \hss}\fi%
\vbox to \cellsz{
\vss
\hbox to \cellsz{\hss$#1$\hss}
\vss}}
\newcommand\tableau[1]{\vtop{\let\\\cr
\baselineskip -16000pt \lineskiplimit 16000pt \lineskip 0pt
\ialign{&\cellify{##}\cr#1\crcr}}}
\newcommand{\kellsize}{24}
\newlength{\kellsz} 
\newsavebox{\kell}
\sbox{\kell}{\begin{picture}(\kellsize,\kellsize)
\put(0,0){\line(1,0){\kellsize}}
\put(0,0){\line(0,1){\kellsize}}
\put(\kellsize,0){\line(0,1){\kellsize}}
\put(0,\kellsize){\line(1,0){\kellsize}}
\end{picture}}
\newcommand\kellify[1]{\def\thearg{#1}\def\nothing{}%
\ifx\thearg\nothing
\vrule width0pt height\kellsz depth0pt\else
\hbox to 0pt{\usebox{\kell} \hss}\fi%
\vbox to \kellsz{
\vss
\hbox to \kellsz{\hss$#1$\hss}
\vss}}
\newcommand\ktableau[1]{\vtop{\let\\\cr
\baselineskip -16000pt \lineskiplimit 16000pt \lineskip 0pt
\ialign{&\kellify{##}\cr#1\crcr}}}

\newcommand{\sellsize}{63}
\newlength{\sellsz} 
\newsavebox{\sell}
\sbox{\sell}{\begin{picture}(\sellsize,20)
\put(0,0){\line(1,0){\sellsize}}
\put(0,0){\line(0,1){\sellsize}}
\put(\sellsize,0){\line(0,1){\sellsize}}
\put(0,\sellsize){\line(1,0){\sellsize}}
\end{picture}}
\newcommand\sellify[1]{\def\thearg{#1}\def\nothing{}%
\ifx\thearg\nothing
\vrule width0pt height\sellsz depth0pt\else
\hbox to 0pt{\usebox{\sell} \hss}\fi%
\vbox to \sellsz{
\vss
\hbox to \sellsz{\hss$#1$\hss}
\vss}}
\newcommand\stableau[1]{\vtop{\let\\\cr
\baselineskip -16000pt \lineskiplimit 16000pt \lineskip 0pt
\ialign{&\sellify{##}\cr#1\crcr}}}

\begin{document}
\pagestyle{plain}
\title{Singularities of Richardson varieties}

\author{Allen Knutson}
\address{Department of Mathematics\\
Cornell University\\
Ithica, NY 14853-4201}
\email{allenk@math.cornell.edu}

\author{Alexander Woo}
\address{Department of Mathematics\\
University of Idaho\\
Moscow, ID 83844-1103 }
\email{awoo@uidaho.edu}

\author{Alexander Yong}
\address{Department of Mathematics\\
University of Illinois at Urbana-Champaign\\
Urbana, IL 61801}
\email{ayong@illinois.edu}



\date{September 14, 2012}

\maketitle

\section{Introduction and the main result}

This paper gives a short proof that essentially all questions concerning singularities of Richardson varieties
reduce to corresponding questions about Schubert varieties.  Consequently,
one quickly deduces new and previously known results.

Let $G$ denote a simple reductive linear algebraic group over an algebraically closed field $\Bbbk$. Fix a choice of
Borel and opposite Borel subgroups $B$ and $B_{-}$, and let $T=B\cap B_{-}$ be the associated maximal torus.
A {\bf generalized flag variety} is a variety $G/P$, where $P\supseteq B$ is a parabolic subgroup.
A $B$-orbit closure $X_{w}:={\overline{BwP/P}}\subseteq G/P$ is known as a {\bf Schubert variety}. Here
$w$ is an element of the Weyl group $W\cong N(T)/T$ of $G$.  If we let $W_P$ be the parabolic
subgroup of $W$ associated to $P$, the choice of a different coset representative in $wW_P$ gives the same Schubert variety.

There has been significant interest in invariants of singularities of
Schubert varieties; two surveys include \cite{Billey.Lakshmibai,
  Brion}.  More generally, one wishes to similarly understand
singularities of other special subvarieties $X$ of various $G/P$.  Two
examples of such subvarieties $X$ are $K$-orbit closures (where $K$ is
a symmetric subgroup of $G$) and Peterson varieties.  Combinatorics
important to the study of Schubert varieties shows up in these cases;
for example, permutation pattern avoidance appears in both the study
of $K$-orbits~\cite{McGovern} and Peterson
varieties~\cite{Insko.Yong}.

This paper concerns {\bf Richardson varieties},
which are the varieties $X_w^v:=X_w\cap X^v$, where
$X^v=\overline{B_{-}vP/P}$ is an {\bf opposite Schubert variety}.
When $v\leq w$ in Bruhat order, $X_w^v$ is nonempty, reduced
and irreducible \cite{Richardson}.  Richardson varieties play a
prominent role in Schubert calculus, as their classes in the
cohomology ring of $G/P$ satisfy $[X_{w}^v]=[X_w]\cdot[X_{w_ov}]$.
Hence one can study the product of Schubert classes and more specifically the {\bf generalized Littlewood-Richardson coefficients} $C_{v,u}^w$ defined by
$[X_u]\cdot[X_v]=\sum_{w} C_{v,u}^w[X_w]$ by studying Richardson varieties. Several recent papers have
studied the singularities of Richardson varieties using various
techniques, including standard monomial theory
\cite{Kreiman.Lakshmibai}, intersection theory \cite{Billey.Coskun,
  Balan}, and Frobenius splitting \cite{Knutson.Lam.Speyer}. Indeed,
\cite{Billey.Coskun, Knutson.Lam.Speyer} study a generalization known as the
\emph{projected Richardson varieties}.

Our main result states that, given any point $x\in X_w^v$, the local
properties of $X_w^v$ at $x$ are completely determined by the Schubert
cell $X_\sigma^\circ$ and opposite Schubert cell $X^\tau_\circ$ to
which $x$ belongs.  Furthermore, for any invariant or property whose
behavior under product of varieties is understood, its behavior on $X_w^v$
at $x$ can be calculated from its behavior on $X_w$ at points of
$X^\tau_\circ$ and its behavior on $X^v$ at points of
$X_\sigma^\circ$.  More precisely:

\begin{Theorem}
\label{thm:main}
Let ${\mathcal P}$ be a local invariant of varieties that is preserved
under products with affine space.  Suppose there exists a function
$f_{\mathcal P}$ such that, for any varieties $X$ and
$Y$,
\begin{equation}
\label{eqn:factor}
{\mathcal P}(x\times y, X\times Y)=f_{\mathcal P}({\mathcal P}(x,X),{\mathcal P}(y,Y)),
\mbox{ \ \ \ \  for any $x\in X$, $y\in Y$.}
\end{equation}
Let $x P\in X^v_w$, and suppose $x P\in (B\sigma P/P)\cap (B_{-}\tau P/P)$
for some $\sigma,\tau\in W$.  Then
\[{\mathcal P}(x P,X^v_w)=f_{\mathcal P}({\mathcal P}(\sigma P,X_w),{\mathcal P}(\tau P, X^v))
=f_{\mathcal P}({\mathcal P}(xP,X_w),{\mathcal P}(xP, X^v)).\]
\end{Theorem}

Our proof gives a local isomorphism, up to a product with affine space,
 between a neighborhood of a point $xP\in X_w^v$ and a product of
local charts on $X_w$ and $X^v$.  This isomorphism is $T$-equivariant. Furthermore, the reader can see from the proof 
that our theorem applies not only to ($T$-equivariant) intrinsic
invariants, but also to local invariants which are relative to the embedding of $X_w^v$
in $G/P$ (for example local $T$-equivariant cohomology classes).

We emphasize that in Theorem~\ref{thm:main}, $xP$ need not be a
$T$-fixed point. For Schubert varieties, the elements of $B$ provide
local isomorphisms between any point $xP\in X_w$ and a $T$-fixed
point; thus one can assume $xP\in (G/P)^T$ if one is interested in
local properties.  However, in general the largest subgroup of $G$ that fixes
$X_w^v$ is $T$, which is not large enough to put every point in the
orbit of a $T$-fixed point.

Call an invariant ${\mathcal P}$ {\bf factorizable} with respect to
$f_{\mathcal P}$ if it satisfies (\ref{eqn:factor}).  Almost all
interesting local invariants are factorizable.  We now mention three
important previously studied special cases of our theorem.

\begin{Example}
\label{exa:Pissmooth}
${\mathcal P}$=``is smooth'' is factorizable with respect to $f_{\mathcal P}=$``logical {\tt and}''.\qed
\end{Example}

Let ${\rm Singlocus}(X)$ denote the singular locus of a variety $X$.
The following was first stated in \cite[Corollary 2.9]{Billey.Coskun}.
 Their proof invokes
Kleiman's transversality theorem; it is also immediate from Theorem~\ref{thm:main} combined with
Example~\ref{exa:Pissmooth}:
\begin{Corollary}
\label{cor:Pissmooth}
${\rm Singlocus}(X_{w}^v)=({\rm Singlocus}(X_w)\cap X^v)\cup (X_w\cap {\rm Singlocus}(X^v))$.
\end{Corollary}


\begin{Example}
\label{exa:PisCM}
${\mathcal P}$=``is normal and Cohen--Macaulay with rational singularities'' is also factorizable with respect to $f_{\mathcal P}=$``logical {\tt and}''.\qed
\end{Example}

It is well known that every Schubert variety is normal and Cohen--Macaulay with rational singularities; see~\cite{Brion} for proofs and historical remarks.  The following was originally proved in~\cite{Brion.paper} in characteristic zero
and~\cite[Appendix A]{Knutson.Lam.Speyer} in general.  Their proofs use a generalization of the Bott--Samelson resolution and, in the latter case, the Frobenius splitting of $G/P$.

\begin{Corollary}
The Richardson variety $X_w^v$ is Cohen-Macaulay and normal with rational singularities.
\end{Corollary}

\begin{Example}
 Consider the ${\mathbb Z}$-graded Hilbert series of ${\rm gr}_{{\mathfrak m}_p}{\mathcal O}_{p,Z}$, the associated graded ring of the local ring ${\mathcal O}_{p,Z}$, which is denoted ${\rm Hilb}({\rm gr}_{{\mathfrak m}_p}{\mathcal O}_{p,Z},q)$.
The {\bf $H$-polynomial} $H_{p,Z}(q)$ is defined by
     \[{\rm Hilb}({\rm gr}_{{\mathfrak m}_p}{\mathcal O}_{p,Z},q)=\frac{H_{p,Z}(q)}{(1-q)^{\dim Z}},\]
     and $H_{p,Z}(1)$ is the {\bf Hilbert-Samuel multiplicity} ${\rm mult}_{p,Z}$.
Let ${\mathcal P}$ be either
$H_{p,Z}(q)$ or ${\rm mult}_{p,Z}$, and $f_{\mathcal P}$=``(usual) multiplication''.
Since ``taking associated graded commutes with tensor product,''
${\mathcal P}$ is factorizable with respect to $f_{\mathcal P}$.
\qed
\end{Example}

     \begin{Corollary}
     \label{Cor:Hfactors}
     Let $xP\in X_{w}^v$. Then $H_{xP,X_{w}^v}(q)=H_{xP,X_w}(q)\cdot H_{xP,X^v}(q)$.
     \end{Corollary}

     In \cite{Li.Yong}, the $H$-polynomial was studied for Schubert varieties, where it was
     conjectured $H_{vP,X_w}(q)\in {\mathbb N}[q]$. Corollary~\ref{Cor:Hfactors}
     extends the conjecture to Richardson varieties.

The consequence that
\begin{equation}
\label{eqn:multfactors}
{\rm mult}_{xP,X_{w}^v}={\rm mult}_{xP,X_w}\cdot {\rm mult}_{xP,X^v}
\end{equation}
also appears to be new. Previously, (\ref{eqn:multfactors})
    was proved for \emph{minuscule} $G/P$ in \cite[Remark 7.6.6]{Kreiman.Lakshmibai}
when $xP\in (G/P)^T$ (which is the case $\sigma=\tau$).
     This result was generalized
by M.~Balan \cite{Balan} who proved (\ref{eqn:multfactors}) for all points $xP$, again assuming that $P$ is minuscule.
(A different generalization in the minuscule setting appears in \cite[Remark~2.15]{Billey.Coskun}.)

In Section~2, we prove Theorem~\ref{thm:main}.  Our proof is
elementary, at least if one accepts standard algebraic groups
language.  It uses a variation of \cite[Lemma~A.4]{Kazhdan.Lusztig},
that relies on a variation on standard results about unipotent groups
found, for example, in~\cite[Section 28.1]{Humphreys}.  We also
include a variant of the proof that avoids the language of algebraic groups (valid only in type $A$,
but otherwise essentially the same as the general proof) in Section~3.
In that section we also discuss some further consequences of
Theorem~\ref{thm:main}.

\section{Proof of Theorem~\ref{thm:main}}

The projection map $\rho:G/B\to G/P$ is a fibration with fibers
locally isomorphic to affine space.  Thus, since our local invariant
is constant under product with affine space, its value at a point
$xP\in {\overline{BwP/P}}\cap {\overline{B_{-}vP/P}}$ is the same as
that for any point in $\rho^{-1}(xP)\subseteq {\overline{Bw^PB/B}}\cap
{\overline{Bv_P B/B}}$, where $w^P$ is the maximal length coset
representative of $w \in W/W_P$ and $v_P$ is the minimal length coset
representative of $v\in W/W_P$.  Therefore, it follows that to check
correctness of Theorem~\ref{thm:main}, we can assume $P=B$.

Let $X_{u}^{\circ}=BuB/B$ and $X^u_{\circ}=B_{-}uB/B$ be the Schubert and
opposite Schubert cells for $u$.
The proof depends on the following version of~\cite[Lemma~A.4]{Kazhdan.Lusztig}.

\begin{Lemma}
\label{lemma:key}
Given any $u\in W$, there exists an isomorphism
\[\eta=(\eta_1,\eta_2): uX_{\circ}^{{\rm id}} \rightarrow X_\circ^{u}
\times X^{\circ}_u\]
such that, for any $x\in uX_\circ^{\id}$, $\eta_1(x)$ and
$x$ are in the same Schubert cell, and $\eta_2(x)$ and $x$ are in the
same opposite Schubert cell.
\end{Lemma}

The result \cite[Lemma~A.4]{Kazhdan.Lusztig}
states that there is such a $T$-equivariant
isomorphism but does not state the additional properties about it
asserted in Lemma~\ref{lemma:key}.

Our proof of Lemma~\ref{lemma:key} depends on the following lemma
about algebraic groups.

\begin{Lemma}
\label{lemma:unipotent}
Let $U$ be a unipotent group on which $T$ acts without nontrivial fixed points, while $\mathfrak{u}=Lie(U)$ is the direct sum of $1$-dimensional eigenspaces for $T$ corresponding to characters $\alpha$ with distinct connected kernels $T_\alpha$.  Furthermore, let
$U_1$ and $U_2$ be $T$-stable subgroups of $U$ such that $U_1\cap
U_2=\{\id\}$, $U_1U_2=U$, and the corresponding Lie algebras satisfy
$\mathfrak{u}=\mathfrak{u_1}\oplus\mathfrak{u_2}$.  Then there exists
an isomorphism of varieties
$$\sigma: U \rightarrow U_1\times U_2$$ such that the factors $\sigma_1$ and $\sigma_2$ satisfy $\sigma_1(u)\in U_2u$ and $\sigma_2(u)\in U_1u$.
\end{Lemma}

Note that $\sigma$ will not be a group homomorphism unless $U_1$ and
$U_2$ commute.

\begin{proof}
It is a standard fact (see for example \cite[Sect. 28.1]{Humphreys})
that the product morphisms $\pi: U_1\times U_2\rightarrow U$ and
$\kappa: U_1 \times U_2\rightarrow U$ given by $\pi(u_1,u_2)=u_1u_2$
and $\kappa(u_1,u_2)=u_2u_1$ are bijective morphisms of varieties. The hypothesis
on the Lie algebras implies they are actually isomorphisms (of
varieties) since the differential is everywhere injective; see the
local isomorphism criterion, for example in~\cite[Theorem~14.9]{Harris}.
However, the inverses of these two maps each only satisfy one of our two
requirements that $\sigma_1(u)\in U_2u$ and $\sigma_2(u)\in U_1u$.

Let $\pi^{-1}_2$ and $\kappa^{-1}_1$ be respectively the inverses of
$\pi$ and $\kappa$ followed by projection onto respectively the second
and first factors.  We now have a map $\sigma=\kappa^{-1}_1\times\pi^{-1}_2:
U\rightarrow U_1\times U_2$.  It remains to prove that
$\kappa^{-1}_1\times\pi^{-1}_2$ is an isomorphism (of schemes), as
once that is done, $\sigma$ is our desired map, as follows.  Since
$\sigma_1=\kappa^{-1}_1$, so $\sigma_1(u)\in U_2u$, and since
$\sigma_2=\pi^{-1}_2$, so $\sigma_2(u)\in U_1u$.

We prove $\kappa^{-1}_1\times\pi^{-1}_2$ is an isomorphism by
induction on the dimension of $U$.  Since $U$ is unipotent, it is
nilpotent and thus has a nontrivial center $U^\prime$
\cite[Lemma~17.4(c)]{Humphreys}.  This center is clearly $T$-stable.
Therefore $U_1^\prime=U_1\cap U^\prime$ and $U_2^\prime=U_2\cap
U^\prime$ are $T$-stable subgroups of $U'$.

As $U^\prime$, $U_1^{\prime}$ and $U_2^{\prime}$ are unipotent and $T$-stable,
they are spanned by products of root subgroups $U_{\alpha}$
\cite[Proposition~28.1]{Humphreys}. These root subgroups are a subset of those
spanning $U$, $U_1$ and $U_2$ respectively.
Hence it follows $U_1^\prime
U_2^\prime=U^\prime$.  Therefore, $\pi$ and $\kappa$ restrict to
isomorphisms ${\widetilde \pi},{\widetilde \kappa}:U_1^\prime\times U_2^\prime\rightarrow U^\prime$.  Since
$U_1^\prime$ and $U_2^\prime$ commute, $\widetilde{\kappa}^{-1}_1\times\widetilde{\pi}^{-1}_2$
is actually the same as both $\widetilde{\kappa}^{-1}$ and $\widetilde{\pi}^{-1}$ on $U^\prime$
and hence invertible.

Furthermore, because $U^\prime$ is central, $\pi$ and $\kappa$ induce
maps $$\overline{\pi}, \overline{\kappa}: U_1/U_1^\prime \times
U_2/U_2^\prime \rightarrow U/U^\prime.$$ (It is straightforward to
check the maps are well-defined.)  By induction, the
map $$\overline{\kappa}^{-1}_1\times\overline{\pi}^{-1}_2:U/U^\prime
\rightarrow U_1/U_1^\prime\times U_2/U_2^\prime$$ is an isomorphism.
(Technically, the induction applies to the natural embeddings, which
are isomorphisms onto their images, of $U_1/U_1^\prime$ and
$U_2/U_2^\prime$ in $U/U^\prime$.)

Summarizing, we have maps in a commutative diagram:
\[\begin{CD}
0   @>>> U' @>>> U @>>> U/U' @>>> 0\\
@. @VV{{\widetilde{\kappa}}^{-1}_1\times{\widetilde{\pi}}^{-1}_2}V     @VV{\sigma=\kappa_1^{-1}\times \pi_2^{-1}}V @VV{\overline{\kappa}^{-1}_1\times\overline{\pi}^{-1}_2}V \\
0 @>>> U_1'\times U_2' @>>> U_1\times U_2 @>>> U_1/U_1' \times U_2/U_2' @>>> 0
\end{CD}\]
It is easy to check from the definitions that $\sigma$ not only restricts to a group homomorphism on $U'$ but moreover that $\sigma(ab)=\sigma(a)\sigma(b)$
whenever $a$ or $b$ is in $U'$. Using this, a straightforward diagram chase similar to the proof
of the five-lemma implies $\sigma$ is a bijection. 

By our hypotheses, the equality of tangent spaces
$T_{p}(U)={\mathfrak u}={\mathfrak u}_1\oplus {\mathfrak u}_2=T_{(a,b)}(U_1\times U_2)$
holds when $p,a,b$ are the identity. Hence $d\sigma$ is an injection
at the identity.  However, injectivity holds everywhere else since
$\sigma$ is $T$-equivariant, the identity is the only $T$-fixed point,
and failure for the differential to be injective is a closed
condition.
Therefore $\sigma$ is an isomorphism of schemes, by another use of
the local criterion for isomorphism.
\end{proof}

\noindent
\emph{Proof of Lemma~\ref{lemma:key}:} Let $U_+\subset B$ and $U_-\subset B_-$ respectively denote the
unipotent subgroups of the Borel and opposite Borel subgroups of $G$.  Given $u\in W$, we
use Lemma~\ref{lemma:unipotent} in the case where $U=uU_-u^{-1}$,
$U_1=U_-\cap uU_-u^{-1}$, and $U_2=U_+\cap uU_-u^{-1}$.
Lemma~\ref{lemma:unipotent} gives us an isomorphism
$$\sigma: U\rightarrow U_1\times U_2.$$ such that $\sigma_1(u)$ and
$u$ differ by left multiplication by an element of $U_2$ and
$\sigma_2(u)$ and $u$ differ by left multiplication by an element of
$U_1$.

Now note that we have isomorphisms $m_0: U\rightarrow uX_{\circ}^{\id}$, $m_1: U_1\rightarrow
X_\circ^{u}$, and $m_2:U_2\rightarrow X^\circ_u$ each defined by $m_i(g)=guB$.  Let $\eta=(m_1\times m_2) \circ \sigma \circ m_0^{-1}$.
Since by Lemma~\ref{lemma:unipotent}, $\sigma_1(g)\in U_2g$ for any $g\in U$ and $U_2\subseteq U_+\subset B$, $\eta_1(guB)\in BguB$ and hence this point is in the same
 Schubert cell as $guB$. Similarly since $U_1\subset B_-$, the points $\eta_2(guB)$ and $guB$ are in the same opposite Schubert cell.
\qed

\noindent
\emph{Proof of Theorem~\ref{thm:main}:}
Let $xB\in G/B$.  Since the charts $\{uX_{\circ}^{\id}\}_{u\in W}$ cover $G/B$, we
can fix $u\in W$ such that $xB\in uX_\circ^{\id}$.  Also let $\eta$ be the map from Lemma~\ref{lemma:key}.

We have the Bruhat and opposite Bruhat decompositions
\begin{equation}
\label{eqn:Bruhat}
X_{w}=\coprod_{\sigma\leq w}X_{\sigma}^{\circ} \mbox{\ \ and \ \ }
X^v=\coprod_{\tau\geq v}X^{\tau}_{\circ}.
\end{equation}
Therefore by (\ref{eqn:Bruhat}), $xB\in X^v_w$ if and only if $xB\in
X^\circ_\sigma$ for some fixed $\sigma\leq w$ and $x\in X_\circ^\tau$
for some fixed $\tau\geq v$.  By Lemma~\ref{lemma:key} we know
$\eta_1(xB)$ and $xB$ are in the same Schubert cell, and we also know
$\eta_2(xB)$ and $xB$ are in the same opposite Schubert
cell. Therefore, $xB\in X^v_w$ if and only if $\eta_1(xB)\in X_w$ and
$\eta_2(xB)\in X^v$.

Therefore, $\eta$ restricts to an isomorphism
\[\eta\mid_{X^v_w}: uX_\circ^{\id}\cap X^v_w \rightarrow (X_\circ^u\cap
X_w) \times (X^\circ_u\cap X^v).\] 
This isomorphism is
scheme-theoretic since both sides are reduced.  (The righthand side is reduced since it is
a product of two integral schemes over their base field and thus integral.)

Now suppose $xB\in X^v_w$.  Since ${\mathcal P}$ is a local invariant,
\[{\mathcal P}(xB,X^v_w)={\mathcal P}(xB,uX_\circ^{\id}\cap X_{w}^{v}).\]
Since $\eta$ is an isomorphism,
\[{\mathcal P}(xB,uX_\circ^{\id}\cap X_{w}^v)={\mathcal P}(\eta(xB),(X_\circ^u\cap X_w)\times (X^\circ_u\cap X^v)).\]
By the hypothesis that ${\mathcal P}$ factorizes with respect to $f_{\mathcal P}$,
\[{\mathcal P}(\eta(xB),(X_\circ^u\cap X_w)\times (X^\circ_u\cap X^v))= f_{\mathcal P}({\mathcal P}(\eta_1(xB),X_\circ^u\cap
X_w),{\mathcal P}(\eta_2(xB),X^\circ_u\cap X^v)).\]
If $xB\in X^\circ_\sigma\cap
X_\circ^\tau$, then by Lemma~\ref{lemma:key} we know $\eta_1(xB)\in X^\circ_\sigma$ and $\eta_2(xB)\in
X_\circ^\tau$.

Since in a Schubert variety every point is locally isomorphic (by the action of $B$)
to the $T$-fixed point in its Schubert cell, we have:
\[{\mathcal P}(\eta_1(xB),X_\circ^u\cap
X_w)={\mathcal P}(\eta_1(xB), X_w)={\mathcal P}(\sigma B,X_w).\]
The first equality is uses \cite[Lemma~A.4]{Kazhdan.Lusztig} (of which we have just proved a stronger version), which states that $(X_{\circ}^u\cap X_w)\times \Bbbk^{\ell(u)}\cong uX_\circ^{\id}\cap X_w$.
We are also using the assumption that ${\mathcal P}$ is invariant under Cartesian product with affine space.

The same arguments apply to the opposite Schubert varieties (using the action of $B_{-}$); hence:
\[{\mathcal P}(\eta_2(xB),X^\circ_u\cap
X^v)={\mathcal P}(\eta_2(xB), X^v)={\mathcal P}(\tau B, X^v).\]
Combining the above we obtain
\[{\mathcal P}(xB,X^v_w)=f_{\mathcal P}({\mathcal P}(\sigma B,X_w),{\mathcal P}(\tau B, X^v)),\]
as desired.\qed

\section{Further consequences and comments}

\subsection{Other singularity invariants}
For brevity, we refer the reader to \cite{WYII} for discussion of
properties of Schubert varieties of interest to us. The purpose of
this section is to explain the extension of this discussion to
Richardson varieties.

There are two kinds of invariants ${\mathcal P}$ we are interested in. The first is of the ``yes/no'' kind. As in Corollary~\ref{cor:Pissmooth}, they are factorizable by $f_{\mathcal P}$=``logical {\tt and}''. If ${\rm non{\mathcal P}locus}(X)$ is the set of points in $X$ where ${\mathcal P}$ takes on the value ``no'', then we have
\begin{equation}
\label{eqn:generalyesno}
{\rm non{\mathcal P}locus}(X_{w}^v)=({\rm non{\mathcal P}locus}(X_w)\cap X^v)\cup (X_w\cap {\rm non{\mathcal P}locus}(X^v)).
\end{equation}

In \cite{WYI} we determined which Schubert varieties are Gorenstein. The property ${\mathcal P}=$``Gorenstein''
is factorizable with respect to $f_{\mathcal P}=$``logical {\tt and}'' since being Gorenstein is a homological
property (and the total complex of the double complex formed by tensoring two free resolutions is exact in this case). Thus, we record:

\begin{Corollary}
\label{Cor:Gor}
(\ref{eqn:generalyesno}) holds for ${\mathcal P}$=``is Gorenstein''.
\end{Corollary}

In \cite{WYI} and \cite{WYII} we conjecture a description of
the non-Gorenstein locus of any Schubert variety. This conjecture
was stated using interval pattern avoidance in the case of $GL_n/B$, furthermore, the conjecture was made explicit and
proved for minuscule $G/P$ by N.~Perrin \cite{Perrin}. In such cases, one can combine these combinatorial descriptions with Corollary~\ref{Cor:Gor}
to obtain descriptions of ${\rm nonGorlocus}(X_{w}^v)$.
Recently, C.~Darayon \cite{Darayon} determined which Richardson varieties in
Grassmannians are \emph{arithmetically} Gorenstein under the standard embedding.
(Theorem~\ref{thm:main} does not apply to the cone singularity.)

Unfortunately, even Corollary~\ref{Cor:Gor} combined with the results in \cite{WYI} does not provide a characterization for when $X_{w}^v\subseteq GL_n/B$
is Gorenstein. One seems to need a (compatible) characterization for the non-Gorenstein locus of a Schubert variety in general.

A stronger property that Gorensteinness is that of ${\mathcal P}$=``is a local complete intersection (lci)''. This ${\mathcal P}$ factorizes with respect to $f_{\mathcal P}=$``logical {\tt and}'', and thus (\ref{eqn:generalyesno}) again applies.
 Recently, a characterization of which $X_w\subseteq GL_n/B$ are local complete intersections has been determined by H.~\'{U}lfarsson and the second author \cite{UlfWoo}.  To further determine when $X^v_w$ is a local complete intersection, one also seems to need a characterization of the 
lci locus of a Schubert variety.

    The second type of invariant we are interested in takes values in ${\mathbb Z}$
    or some (Laurent) polynomial ring. Suppose ${\mathcal P}$ is factorizable with respect
to $f_{\mathcal P}=$``usual multiplication of numbers or (Laurent) polynomials''. If we
let ${\mathcal P}_{x,X}$ denote the value taken at $x\in X$, then just as in Corollary~\ref{Cor:Hfactors}, we have
\begin{equation}
\label{eqn:numfactor}
{\mathcal P}_{xP,X_{w}^v}={\mathcal P}_{xP,X_w}\cdot {\mathcal P}_{xP,X^v}.
\end{equation}
 for any $xP\in X_{w}^v$.

More refined than knowing if a point $x\in X$ is Gorenstein is to know the \emph{Cohen-Macaulay type} of that point.
Let ${\mathcal P}_{x,X}=CM_{x,X}\in {\mathbb N}$ denote this statistic. (When $CM_{x,X}=1$ then $x$ is a Gorenstein point, assuming $X$ is Cohen-Macaulay, which
is always true in our situation.) Again, since ${\mathcal P}$ is a homological property, it factorizes with respect to $f_{\mathcal P}=$``(usual multiplication)''. Thus we obtain another example of (\ref{eqn:numfactor}):

     \begin{Corollary}
     \label{Cor:CMfactors}
     (\ref{eqn:numfactor}) holds for ${\mathcal P}=$``Cohen-Macaulay type''.
     \end{Corollary}

Let $P_{v,w}(q)\in {\mathbb N}[q]$ denote a Kazhdan-Lusztig
polynomial. Owing to its interpretation as a Poincar\'{e} polynomial
for local intersection cohomology at $vP$ in $X_w$ (together with a
Kunneth-type formula for local intersection cohomology), one can apply
(\ref{eqn:numfactor}). Work of L.~Li and the second author
\cite{Li.Yong} suggests an analogy between the polynomials $P_{v,w}(q)$ and
$H_{vP,X_w}(q)$. In particular, for Grassmannians and covexillary
Schubert varieties, it is shown that $P_{v,w}(q)\preceq
H_{vP,X_w}(q)$, where $\preceq$ means ``coefficient-wise
$\leq$''. This therefore lifts to Richardson varieties.

B.~Wyser \cite{Wyser} has recently proved that certain $K$-orbits (for
example, certain $K=GL_p\times GL_q$ orbits in $GL_n/B$) are
isomorphic to Richardson varieties. Therefore our results apply to
those $K$-orbits.  Further discussion may appear elsewhere.



\subsection{Another proof of Lemma~\ref{lemma:key}, when in Type~$A$}
Assume $x$ is the generic matrix in $uX_{\circ}^{\rm id}$: the
rows are a generic matrix in $U_{-}$, but are $u$-permuted.
Let $z_{ij}$ be the entry in row $i$ and column $j$. Note $u(i)=j$ means there is
a $1$ in column $i$ and row $j$ of the matrix. See Example~\ref{exa:31542}.

Define $\eta_1(x)$ to be
obtained by ``sweeping up using the $1$'s in each column'': take the
column $j_1$
whose ``$1$'' (coming from $u$) is in row $n$ (so $j_1=u^{-1}(n)$) and do upward row operations so that all entries
in column $j_1$ strictly 
above row $n$ are $0$. Then find the column $j_2$ whose ``$1$'' is
in row $n-1$ and do operations that place $0$ above that ``$1$'' etc. Declare the result
after sweeping in all $n$ columns to be $\eta_1(x)$. (The resulting matrix can be obtained
by any sequence of row operations that put $0$'s above each ``$1$''; however
the stated order is most efficient, as no column has to be swept more than once.)
Clearly $\eta_1(x)$ is in $BuB/B$ and is in the same Schubert cell as $x$ since the former
comes from the latter by left multiplication of a matrix from $B$.

Similarly, define $\eta_2(x)\in B_{-}uB/B$ by downward sweeping operations,
starting with the column whose ``$1$'' (from $u$) appears in row $1$, etc. Likewise,
$\eta_2(x)$ and $x$ are in the same opposite Schubert cell.

Let
\[D^{\rm up}=\{(i,j)\in [n]\times [n]: i>\pi(j) \mbox{\ and \ } j<\pi^{-1}(i)\}\]
\[D^{\rm down}=\{(i,j)\in [n]\times [n]: i<\pi(j) \mbox{ \ and \ } j<\pi^{-1}(i)\}.\]
These are the positions in $\eta_1(x)$ and $\eta_2(x)$ respectively not (\emph{a priori}) equal to $0$ or $1$.

\begin{Claim}
\label{claim:A}
(i) In $\eta_1(x)$, the entry in position $(i,j)\in D^{\rm up}$ is an expression of the form $z_{ij}+f$ where
$f$ is a polynomial in the generic entries $z_{a,b}$ where either $b=j$ and $a>i$, or
$b>j$.

(ii) Similarly, in $\eta_2(x)$, the entry in position $(i,j)\in D^{\rm down}$ is an expression of the form $z_{ij}+g$ where
$g$ is a polynomial in the generic entries $z_{a,b}$ where either $b=j$ and $a<i$, or
$b>j$.
\end{Claim}
\begin{proof}
The argument for (i) is the upside down version of (ii). Let us argue the latter case (where one does downward sweeping). The only $1$'s we use to sweep that affect some position
$(i,j)\in D^{\rm down}$ are those strictly northeast of $(i,j)$. Say this $1$ is in position $(i',j')$ (with $i'<i$ and $j'>j$). Then the row operation involving this $1$ converts the entry $y_{i,j}$ in position $(i,j)$ to $y_{ij}-y_{i',j}y_{i,j'}$,
and the result follows by induction on $j\leq n$.
\end{proof}



We now recover
all generic entries $z_{ij}$ in $x$ from the entries of $\eta_1(x)$ and $\eta_2(x)$.
Inductively, we argue that, if we know all such entries in columns $k,k+1,\ldots,n$ of $x$, we can
deduce the entries in column $k-1$ of $x$.

For the base case of the argument, note that column $n$ of $x$ consists of
only $0$'s and $1$'s, so there is nothing to do. In the induction step, the argument
depends on the relative positions of the $1$'s in columns $k-1$ and $k$:

\noindent
\emph{Case I -- The $1$ in column $k-1$ is higher (lower index) than the one
in column $k$:} Suppose $u(k-1)=b$ and $u(k)=c$ where $b<c$. First consider
positions $(q,k-1)$ for $q<b$. We use $\eta_2(x)$. Either position $(1,k-1)$
is $0$ in $x$ (because there is a $1$ to the left of column $k-1$ in the first row),
or $z_{1,k-1}$ appears in that position of $\eta_2(x)$ and hence is determined.
Then by Claim~\ref{claim:A}(ii) and induction, working top to bottom,
we determine each $z_{q,k-1}$ for $q<b$.

    Now consider indices $(r,k-1)$ for $r>b$. This time notice position $(n,k-1)$
    is either $0$ in $x$ or $z_{n,k-1}$ appears there in both $x$ and $\eta_1(x)$ and
    hence is determined. Then the remaining generic variables of $x$ in column $k-1$
    are obtained from $\eta_1(x)$ by working bottom to top, applying induction
    and Claim~\ref{claim:A}(i).

\noindent
\emph{Case II -- The $1$ in column $k-1$ is lower (higher index) than the one
in column $k$:} The argument is the upside down version of Case I, reversing
the roles of $\eta_1(x)$ and $\eta_2(x)$ in the obvious way.
\qed

Our second proof of Theorem~\ref{thm:main} (in type $A$)
now concludes as in Section~2.

\begin{Example}
\label{exa:31542}
Let $u=31542\in S_5$. A generic matrix in $uX_\circ^{\id}$ has a standard form
$$x=\begin{pmatrix}
z_{11} & 1 & 0 & 0 & 0 \\
z_{21} & z_{22} & z_{23} & z_{24} & 1 \\
1 & 0 & 0 & 0 & 0 \\
z_{41} & z_{42} & z_{43} & 1 & 0 \\
z_{51} & z_{52} & 1 & 0 & 0
\end{pmatrix}.$$

Then
\[\eta_1(x)=\begin{pmatrix}
0 & 1 & 0 & 0 & 0 \\
0 & z_{22}-z_{24}(z_{42}-z_{52}z_{43})-z_{23}z_{52} & 0 & 0 & 1 \\
1 & 0 & 0 & 0 & 0\\
z_{41}-z_{51}z_{43} & z_{42}-z_{52}z_{43} & 0 & 1 & 0\\
z_{51} & z_{52} & 1 & 0 & 0
\end{pmatrix}
\]
and
\[\eta_2(x)=\begin{pmatrix}
z_{11} & 1 & 0 & 0 & 0 \\
z_{21}-z_{11}z_{22} & 0 & z_{23} & z_{24} & 1 \\
1 & 0 & 0 & 0 & 0 \\
0 & 0 & z_{43} & 1 & 0 \\
0 & 0 & 1 & 0 & 0
\end{pmatrix}.\]

We show how to
recover the entries of $x$ as polynomials of the entries in
$\eta_1(x)$ and $\eta_2(x)$.  The variables $z_{51}$, $z_{52}$, $z_{43}$, $z_{23}$, $z_{24}$,
and $z_{11}$ appear in $\eta_1(x)$ or $\eta_2(x)$.  (These are the
variables which have no $0$'s between them and the $1$ to their right
in $\eta_1(x)$ or $\eta_2(x)$.)  Given these, we can recover $z_{41}$, $z_{42}$,
and $z_{22}$ from the entries of $\eta_1$ and $\eta_2$, since $z_{41}=(z_{41}-z_{51}z_{43})+z_{51}z_{43}$, $z_{42}=(z_{42}-z_{52}z_{43})+z_{52}z_{43}$, and
$z_{22}=(z_{22}-z_{24}(z_{42}-z_{52}z_{43})-z_{23}z_{52})
+z_{23}z_{52}+z_{24}(z_{42}-z_{52}z_{43})$.  Given these, we can recover $z_{21}$, since
$z_{21}=(z_{21}-z_{11}z_{22})+z_{11}z_{22}$.\qed
\end{Example}

\section*{Acknowledgements}
AK was supported by an NSF grant. AY was supported by an NSF grant; he also thanks UIUC's Center for Advanced Study, where he was a 
Beckman Fellow during the completion of this text.

\end{document}